\documentclass[12pt]{amsart}
\usepackage[active]{srcltx}
\usepackage{calc,amssymb, mathrsfs, amsthm, amsmath, amscd, eucal,ulem}
\usepackage{alltt}
\usepackage[left=1.35in,top=1.25in,right=1.35in,bottom=1.25in]{geometry}
\RequirePackage[dvipsnames,usenames]{color}

\input{kmacros3.sty}
\input{xy}
\xyoption{all}
\usepackage{tikz}

\numberwithin{equation}{theorem}

\renewcommand{\m}{\mathfrak{m}}
\renewcommand{\n}{\mathfrak{n}}

\DeclareMathOperator{\depth}{depth}

\DeclareMathOperator{\Coh}{Coh}
\DeclareMathOperator{\QCoh}{QCoh}

\DeclareMathOperator{\sheafhom}{\mathscr{H}{\kern -2pt om}}

\usepackage{fullpage}

\usepackage{setspace}
\usepackage{hyperref}

\usepackage{enumerate}

\usepackage{graphicx}

\usepackage[all,cmtip]{xy}
\usepackage{verbatim}

\theoremstyle{theorem}

\newtheorem{proposition/definition}[theorem]{Proposition/Definition}

\begin{document}
\title{The vanishing conjecture for maps of Tor and derived splinters}
\author{Linquan Ma}
\address{Department of Mathematics\\ University of Utah\\ Salt Lake City\\ Utah 84112}
\email{lquanma@math.utah.edu}
\maketitle
\begin{abstract}
We say an excellent local domain $(S,\n)$  {\it satisfies the vanishing conditions for maps of Tor}, if for every $A\to R\to S$ with $A$ regular and $A\to R$  module-finite torsion-free extension, and every $A$-module $M$, the map $\Tor^A_i(M, R)\to \Tor_i^A(M, S)$ vanishes for every $i\geq 1$. Hochster-Huneke's conjecture (theorem in equal characteristic) thus states that regular rings satisfy such vanishing conditions \cite{HochsterHunekeApplicationsofbigCMalgebras}. The main theorem of this paper shows that, in equal characteristic, rings that satisfy the vanishing conditions for maps of Tor are exactly {\it derived splinters} in the sense of Bhatt \cite{BhattDerivedsplintersinpositivecharacteristic}. In particular, rational singularities in characteristic $0$ satisfy the vanishing conditions. This greatly generalizes Hochster-Huneke's result \cite{HochsterHunekeApplicationsofbigCMalgebras} and Boutot's theorem \cite{BoutotRationalsingularitiesandquotientsbyreductivegroups}. Moreover, our result leads to a new (and surprising) characterization of rational singularities in terms of splittings in module-finite extensions.
\end{abstract}

\section{Introduction}

In \cite{HochsterHunekeApplicationsofbigCMalgebras}, Hochster and Huneke proved the following extremely strong vanishing result in equal characteristic:

\begin{theorem}[{\it cf.} Theorem 4.1 in \cite{HochsterHunekeApplicationsofbigCMalgebras}]
\label{theorem--vanishing theorem for maps of Tor}
Let $A$ be an equal characteristic regular domain, let $R$ be a module-finite and torsion-free extension of $A$, and let $R\to S$ be any homomorphism from $R$ to a regular ring $S$. Then for every $A$-module $M$ and every $i\geq 1$, the map $\Tor_i^A(M, R)\to\Tor^A_i(M, S)$ vanishes.
\end{theorem}

It is also conjectured by Hochster and Huneke that Theorem \ref{theorem--vanishing theorem for maps of Tor} holds in mixed characteristic. This is one of the well-known homological conjectures: {\it the vanishing conjecture for maps of Tor}. The importance of Theorem \ref{theorem--vanishing theorem for maps of Tor}, as well as the corresponding conjecture in mixed characteristic, lies in the fact that, in any characteristic, it implies both the direct summand conjecture and the conjecture that direct summands of regular rings are Cohen-Macaulay \cite{HochsterHunekeApplicationsofbigCMalgebras}. Indeed, it was shown in \cite{Ranganathanthesis} that the vanishing conjecture for maps of Tor is equivalent to a strong form of the direct summand conjecture (we refer to \cite{Ranganathanthesis} for details).

In fact, results very similar to Theorem \ref{theorem--vanishing theorem for maps of Tor} were first proved in \cite{HochsterHunekePhantomhomology}, in characteristic $p>0$ only, using tight closure and phantom homology theory.\footnote{It is pointed out in the introduction of \cite{HochsterHunekePhantomhomology} that by reduction to characteristic $p>0$, one can develop the corresponding theory in characteristic $0$. The full results in \cite{HochsterHunekePhantomhomology} are, in some sense, even stronger than Theorem \ref{theorem--vanishing theorem for maps of Tor}, but are slightly technical to state here. However we point out that all these (stronger) results can be established by the argument used in \cite{HochsterHunekeApplicationsofbigCMalgebras}. Our method can provide generalizations of these results also, both in characteristic $p>0$ and characteristic $0$, see Remark \ref{remark--key theorem extehds original vanishing theorems}.} The proof given in \cite{HochsterHunekeApplicationsofbigCMalgebras} makes crucial use of the existence of weakly functorial balanced big Cohen-Macaulay algebras in equal characteristic. In characteristic $p>0$, the existence of such algebras follows directly from \cite{HochsterHunekeInfiniteintegralextensionsandbigCMalgebras}, where it was shown that the absolute integral closure $R^+$ is such an algebra. In characteristic $0$, the construction of weakly functorial balanced big Cohen-Macaulay algebras depends on a very delicate and difficult reduction to characteristic $p>0$ argument (we refer to Section 3 of \cite{HochsterHunekeApplicationsofbigCMalgebras} for details). In mixed characteristic, the analogy of Theorem \ref{theorem--vanishing theorem for maps of Tor} is known when $A$, $R$, $S$ all have dimension less than or equal to three \cite{HochsterBigCohen-Macaulayalgebrasindimensionthree}, based on Heitmann's results \cite{HeitmannDirectsummandconjectureindimensionthree}. However, in general the vanishing conjecture for maps of Tor is wide open in mixed characteristic.

In this paper, we investigate Theorem \ref{theorem--vanishing theorem for maps of Tor} in some new and different ways. We study the ``converse" of Theorem \ref{theorem--vanishing theorem for maps of Tor} in the following sense: in a given characteristic, for which local domain $S$, the map $\Tor_i^A(M, R)\to\Tor^A_i(M, S)$ vanishes for every $A\to R\to S$ and every $A$-module $M$ (where $A$ is regular and $A\to R$ is a module-finite torsion-free extension)? We will say such $S$ {\it satisfies the vanishing conditions for maps of Tor} (see Section 2 for precise definitions). We will show that, in all characteristics, such vanishing conditions imply $S$ has only pseudo-rational singularities, which is a characteristic-free analogue of rational singularities. Our main result in equal characteristic is the following:

\begin{theorem}[=Theorem \ref{theorem--main theorem}]
\label{theorem--main theorem in equal characteristic}
Let $S$ be a local domain that is essentially of finite type over a field. The following are equivalent:
\begin{enumerate}
\item $S$ satisfies the vanishing conditions for maps of Tor.
\item $S$ is a derived splinter.
\item For every regular local ring $A$ with $S=A/P$ and every module-finite torsion-free extension $A\to B$ with $Q\in \Spec B$ lying over $P$, $P\to Q$ splits as $A$-modules.
\end{enumerate}
\end{theorem}

We note that derived splinters are formally introduced by Bhatt in \cite{BhattDerivedsplintersinpositivecharacteristic}, and are well understood in equal characteristic: they are equivalent to rational singularities in characteristic $0$ \cite{KovacsAcharacterizationofrationalsingularities}, \cite{BhattDerivedsplintersinpositivecharacteristic} and in characteristic $p>0$, they turn out to be the same as splinters \cite{BhattDerivedsplintersinpositivecharacteristic} (see Section 2 for precise definitions for splinters and derived splinters). In fact, at least in characteristic $0$, the idea of derived splinters plays a crucial role in our proofs.

As regular local rings in equal characteristic are derived splinters, Theorem \ref{theorem--main theorem in equal characteristic} $(1)\Leftrightarrow(2)$ greatly extends Theorem \ref{theorem--vanishing theorem for maps of Tor}. We will see in Remark \ref{remark--VCT generalizes Boutot's theorem} that Theorem \ref{theorem--main theorem in equal characteristic} also generalizes Boutot's theorem that direct summands of rational singularities are rational singularities \cite{BoutotRationalsingularitiesandquotientsbyreductivegroups} (Boutot's theorem follows from the vanishing of Tor applied to $M=E_A$, the injective hull of $A$). Moreover, as an immediate consequence of Theorem \ref{theorem--main theorem in equal characteristic} $(2)\Leftrightarrow(3)$, we obtain the following new characterization of rational singularities. We find this characterization surprising as it only addresses splittings in module-finite extensions.

\begin{corollary}[=Corollary \ref{corollary--main corollary}]
\label{corollary--new characterization of rational singularities}
Let $(S,\n)$ be a local domain essentially of finite type over a field of characteristic $0$. Then $S$ has rational singularities if and only if for every regular local ring $A$ with $S=A/P$, every module-finite torsion-free extension $A\to T$, and every $Q\in\Spec T$ lying over $P$, the map $P\to Q$ splits as a map of $A$-modules.
\end{corollary}

This paper is organized as follows. In Section 2, we recall and review the basic theories, and we introduce two important concepts: the vanishing conditions for maps of Tor and the vanishing conditions for maps of local cohomology. The rest of the paper is devoted to the proof of Theorem \ref{theorem--main theorem in equal characteristic}. In Section 3 we show that the vanishing conditions for maps of Tor implies pseudo-rationality. In Section 4 we prove $(1)\Leftrightarrow(3)$ in Theorem \ref{theorem--main theorem in equal characteristic}. Finally, in Section 5 we prove $(1)\Leftrightarrow(2)$ in Theorem \ref{theorem--main theorem in equal characteristic} and we also prove some partial results in mixed characteristic: for example, we show that the vanishing conjecture for maps of Tor implies the derived direct summand conjecture. Throughout this paper, unless otherwise stated, we will make the following assumptions on commutative rings and schemes (we will sometimes repeat and emphasize these conditions):
\begin{enumerate}
\item All rings are Noetherian, excellent and are homomorphic image of regular rings.
\item All schemes are Noetherian, separated, excellent that admit dualizing complexes.
\item In characteristic $0$, all rings and schemes are essentially of finite type over a field.
\end{enumerate}

We point out that (1) and (2) are very mild conditions (e.g., all rings essentially of finite type over a complete local ring satisfy (1)). We make the assumption (3) mainly because we need to apply the Grauert-Riemenschneider type vanishing theorems \cite{GrauertRiemenschneiderVerschwindungssatze}, \cite{KollarHigherdirectimagesofdualizingsheaves1} in characteristic $0$.


\section{Definitions and Preliminaries}

We begin with some basic definitions of plus closure. Let $S$ be an integral domain and $I\subseteq S$ be an ideal. The {\it plus closure} of $I$, $I^+$, is the set of elements $x\in S$ such that $x\in IT$ for some module-finite extension $T$ of $S$. $I$ is called plus closed if $I^+=I$. The {\it absolute integral closure} of $S$, denoted by $S^+$, is the integral closure of $S$ in the algebraic closure of the fraction field of $S$, which is also the direct limit of all the module-finite domain extensions of $S$ \cite{HochsterHunekeInfiniteintegralextensionsandbigCMalgebras}. It follows that $I^+=IS^+\bigcap S$. The plus closure of $0$ in $H_\n^d(S)$, the top local cohomology module, is defined as $0^+_{H_\n^d(S)}=\ker (H_\n^d(S)\to H_\n^d(S^+))$.

A domain $S$ (resp., An integral scheme $X$) is called a {\it splinter}, if for every module-finite extension $T$ of $S$ (resp., every finite surjective map $Y\to X$), the natural map $S\to T$ (resp., $O_X\to O_Y$) is split in the category of $S$-modules (resp., $O_X$-modules). It is easy to see that $S$ is a splinter if and only if every ideal in $S$ is plus closed.

Let $(S,\n)$ be an excellent local domain of characteristic $p>0$. The top local cohomology module $H_\n^{d}(S)$ has a natural Frobenius action. In this situation, there is a unique largest proper submodule of $H_\n^{d}(S)$ that is stable under the Frobenius action, which is $0^*_{H_\n^{d}(S)}$, the tight closure of $0$ in $H_\n^{d}(S)$ \cite{SmithFRatImpliesRat}. $(S,\n)$ is called {\it $F$-rational}, if it is Cohen-Macaulay and $0^*_{H_\n^{d}(S)}=0$ \cite{HochsterHunekeFRegularityTestElementsBaseChange}, \cite{SmithFRatImpliesRat}. This is not the original definition of $F$-rationality, but it turns out to be extremely useful in many applications. It is worth mentioning that a deep result of Smith \cite{Smithtightclosureofparameterideals} shows that $0^*_{H_\n^{d}(S)}=0^+_{H_\n^{d}(S)}$, which we will need in Section 5.

We make some more comments on splinters. In equal characteristic $0$, using the trace map, it is straightforward to check that splinters are exactly normal schemes. However, even in equal characteristic $p>0$ in the affine case, splinters are quite mysterious. It is known that affine splinters in characteristic $p>0$ are always $F$-rational \cite{Smithtightclosureofparameterideals} \cite{BhattDerivedsplintersinpositivecharacteristic}, and it is conjectured that they are {\it $F$-regular}, which is a natural strengthening of $F$-rationality and an important concept in tight closure theory.\footnote{As we will not use deep results in tight closure theory, we omit the precise definition of $F$-regularity (and the original definition of $F$-rationality). We refer to \cite{HochsterHunekeTC1} for details on tight closure theory.} We refer to \cite{SinghQGorensteinsplintersareFregular} and \cite{ChiecchioEnescuMillerSchwedeTestidealswithfinitelygeneratedanticanonical} for the best partial results on this conjecture. In mixed characteristic, our knowledge about splinters is minimal: Hochster's famous {\it direct summand conjecture} asserts that regular local rings are splinters. This conjecture is known in dimension $\leq 3$ \cite{HeitmannDirectsummandconjectureindimensionthree}, and is open (in mixed characteristic) in dimension $\geq 4$.

Following \cite{BhattDerivedsplintersinpositivecharacteristic}, we say an integral scheme $X$ is a {\it derived splinter}, if for any proper surjective map $f$: $Y\to X$, the pullback map $O_X\to \mathbf{R}f_*O_Y$ is split in the derived category $D(\Coh(X))$ of coherent sheaves on $X$. This is the same as requiring $O_X\to \mathbf{R}f_*O_Y$ to split in $D(\QCoh(X))$, the derived category of quasi-coherent sheaves on $X$. It is easy to see that derived splinters are splinters. It was first observed in \cite{KovacsAcharacterizationofrationalsingularities} that derived splinters in characteristic $0$ coincide with rational singularities,\footnote{This was proved in \cite{KovacsAcharacterizationofrationalsingularities} when $Y\to X$ has connected fibres (which was suffices for the applications in \cite{KovacsAcharacterizationofrationalsingularities}). A complete proof was given in Theorem 2.12 in \cite{BhattDerivedsplintersinpositivecharacteristic}.}
while it was shown in \cite{BhattDerivedsplintersinpositivecharacteristic} that, quite surprisingly, derived splinters are equivalent to splinters in characteristic $p>0$.

Next we recall pseudo-rational singularities \cite{LipmanTeissierPseudorationallocalringsandatheoremofBrianconSkoda}: A $d$-dimensional local ring $(R,\m)$ is called {\it pseudo-rational} if it is normal, Cohen-Macaulay, analytically unramified (i.e., the completion $\widehat{R}$ is reduced), and if for every proper, birational map  $\pi$: $W\to \Spec R$ with $W$ normal, the canonical map $H_\m^d(R)\to H_E^d(W, O_W)$ is injective where $E=\pi^{-1}(\m)$ denotes the closed fibre. Pseudo-rationality is a property of local rings which is an analog of rational singularities for more general schemes, e.g., rings which may not have a desingularization. When the ring is essentially of finite type over a field of characteristic $0$, pseudo-rational singularities are the same as rational singularities. In characteristic $p$, pseudo-rationality is slightly weaker than $F$-rationality \cite{SmithFRatImpliesRat}, \cite{HaraRatImpliesFRat}.

We summarize the relations between these concepts. In characteristic $0$, we have:
\[\mbox{derived}~\mbox{splinter}=\mbox{rational}~\mbox{singularities}=\mbox{pseudo-rational}\Longrightarrow \mbox{splinter}.\]
In characteristic $p>0$, we have:
\[\mbox{derived}~\mbox{splinter}=\mbox{splinter}\Longrightarrow F\mbox{-rational}\Longrightarrow\mbox{pseudo-rational}.\]

Now we introduce the central concepts that we will study in this paper:

\begin{definition}
\label{definition--VCT} We say a local domain $(S,\n)$ satisfies the vanishing conditions for maps of Tor, if for every $A\to R\to S$ such that $A$ is a regular domain, $A\to R$ is a module-finite torsion-free extension, and $A$, $R$, $S$ have the same characteristic,\footnote{This means $A, R, S$ all have equal characteristic, i.e., they all contain a field, or they all have mixed characteristic (i.e., the characteristic of the ring is different from that of its residue field).} the natural map $\Tor_i^A(M, R)\to\Tor_i^A(M, S)$ vanishes for every $A$-module $M$ and every $i\geq 1$.
\end{definition}

It is also quite natural to ask that: if $(R,\m)\twoheadrightarrow (S,\n)$ is a surjection of local domains, when does $H_\m^j(R)\to H_\n^j(S)$ vanish for every $j<\dim R$? (this is inspired by Corollary 4.24 of \cite{HochsterHunekePhantomhomology}, which itself is a consequence of Theorem \ref{theorem--vanishing theorem for maps of Tor}). Hence similar to the vanishing conditions for maps of Tor, we want to introduce certain vanishing conditions for maps of local cohomology. Since there are several equivalent ways to define this, we summarize them into a proposition.

\begin{proposition/definition}
\label{definition--VCLC}
Let $(S,\n)$ be a local domain of dimension $d$. Then the following are equivalent (we always assume $R$, $S$ have the same characteristic):
\begin{enumerate}
\item For every surjection $(R,\m) \twoheadrightarrow (S,\n)$ with $(R,\m)$ equidimensional, the induced map $H_\m^j(R)\to H_\n^j(S)$ vanishes for every $j<\dim R$.
\item For every surjection $(R,\m) \twoheadrightarrow (S,\n)$ with $(R,\m)$ a local domain, the induced map $H_\m^j(R)\to H_\n^j(S)$ vanishes for every $j<\dim R$.
\item $S$ is Cohen-Macaulay and for every surjection $(R,\m) \twoheadrightarrow (S,\n)$ such that $(R,\m)$ is a local domain with $\dim R>d$, the induced map $H_\m^d(R)\to H_\n^d(S)$ vanishes.
\item $S$ is Cohen-Macaulay and for every surjection $(R,\m) \twoheadrightarrow (S,\n)$ such that $\dim R/P>d$ for every minimal prime of $P$ of $R$,  the induced map $H_\m^d(R)\to H_\n^d(S)$ vanishes.
\end{enumerate}
We say $(S,\n)$ satisfies the vanishing conditions for maps of local cohomology, if it satisfies the above equivalent conditions.
\end{proposition/definition}

\begin{proof} $(1)\Rightarrow(2)$: This is obvious.

$(2)\Rightarrow(3)$: Applying $(2)$ to $R=S$, we get that the identity map $H_\n^j(S)\to H_\n^j(S)$ vanishes for every $j<\dim S=d$. Thus $S$ is Cohen-Macaulay. The remaining part is obvious (note that one cannot apply $(3)$ to $R=S$, because the hypothesis on $R$ in $(3)$ forces $\dim R>d$).

$(3)\Rightarrow(4)$: Since $S$ is a domain, every surjection $R\twoheadrightarrow S$ factors through $R\twoheadrightarrow R'\twoheadrightarrow S$, where $R'=R/P$ for some minimal prime $P$ of $R$. Now $(3)$ implies $H_\m^d(R')\to H_\n^d(S)$ vanishes because $\dim R'=\dim R/P>d$. Thus $H_\m^d(R)\to H_\n^d(S)$ also vanishes.

$(4)\Rightarrow(1)$: If $\dim R=\dim S=d$ (i.e., $R=S$) in $(1)$, then $H_\m^j(R)\to H_\n^j(S)$ vanishes for every $j<\dim R=d$ because $H_\n^j(S)=0$ ($S$ is Cohen-Macaulay). Otherwise we have $\dim R>d$. Since $R$ is equidimensional, $\dim R/P>d$ for every minimal prime $P$ of $R$. Thus applying $(4)$, we know that $H_\m^d(R)\to H_\n^d(S)$ vanishes.
\end{proof}

\begin{remark}
One cannot expect that $H_\m^d(R)\to H_\n^d(S)$ vanish for all $R\twoheadrightarrow S$ with $\dim R>d$, even when $S$ is regular. For example, let $R=\frac{k[[x,y,z]]}{(x,y)\bigcap(z)}$ and $S=k[[z]]$. We know that $\dim R=2$ and $\dim S=1$. But it is easy to check that $H_\m^1(R)\to H_\n^1(S)$ is surjective and hence does not vanish. The trouble here is that there is a component of $R$ that has the same dimension as $S$. Thus the hypotheses in Definition \ref{definition--VCLC} (1)--(4) are necessary.
\end{remark}

We will see in later sections that the vanishing conditions for Tor and for local cohomology are deeply related: Proposition \ref{proposition--VCT implies VCLC and pseudo-rational}, Theorem \ref{theorem--main theorem on VCLC}.



\section{Vanishing of Tor, vanishing of local cohomology and pseudo-rationality}

In this section we will show that the vanishing conditions for maps of Tor implies pseudo-rationality, which will be a crucial ingredient in proving $(1)\Rightarrow(2)$ in Theorem \ref{theorem--main theorem in equal characteristic}. We also obtain many characteristic-free results of independent interest.

\begin{lemma}
\label{lemma--image contains the plus closure}
Let $(S,\n)$ be a local domain that is a homomorphic image of a regular ring. Then we have:
\begin{equation}
\label{equation--comparison of the sum of image with plus closure}
\sum_{R}\im(H_\m^d(R)\to H_\n^d(S))\supseteq 0^+_{H_\n^d(S)}
\end{equation}
where the sum is taken over all $R\twoheadrightarrow S$ such that $\dim R/P>\dim S=d$ for every minimal prime $P$ of $R$. In particular, if $(S,\n)$ satisfies the vanishing conditions for maps of local cohomology, then we have $0^+_{H_\n^d(S)}=0$.
\end{lemma}
\begin{proof}
Since $0^+_{H_\n^d(S)}=\ker(H_\n^d(S)\to H_\n^d(S^+))=\cup_T \ker(H_\n^d(S)\to H_\n^d(T))$ where $T$ runs over all module-finite domain extensions of $S$. It suffices to show that
$\sum_{R}\im(H_\m^d(R)\to H_\n^d(S))\supseteq \ker(H_\n^d(S)\to H_\n^d(T))$ for every such $T$.

We write $S=A/P$ for some regular local ring $A$ such that $\dim A\geq d+1$. Let $t_1,\dots,t_n$ be a set of generators of $T$ over $S$. Since $T$ is integral over $S$, each $t_i$ satisfies a monic polynomial $f_i$ over $S$. We can lift each $f_i$ to $A$ and form the ring $B=\frac{A[x_1,\dots,x_n]}{(f_1,\dots,f_n)}$. We have a natural surjective map $B\twoheadrightarrow T$ with kernel $Q\in\Spec{B}$. It is clear that $Q$ lies over $P$ in $A$. Let $R=A+Q\subseteq B$. We know that $R/Q=A/P=S$. In sum, we have:
\[  \xymatrix{
   0 \ar[r] & Q  \ar[r] &    B  \ar[r]  & T \ar[r] & 0\\
   0 \ar[r] & Q  \ar[r]\ar[u]^\cong &    R  \ar[u] \ar[r]  & S \ar[r]\ar[u] & 0
} .\]
Let $\m$ be the pre-image of $\n$ in $R$. Because $B$ is free over $A$ and $R$ is a subring of $B$, $R$ is torsion-free over $A$. Now localizing at $\m$ if necessary, we know that $(R,\m)$ is equidimensional and $\dim R=\dim B=\dim A\geq d+1$. This guarantees that $\dim R/P>d$ for every minimal prime $P$ of $R$. The induced long exact sequences on local cohomology gives:
\[  \xymatrix{
    H_\m^d(B) \ar[r] & H_\m^d(T)=H_\n^d(T)  \ar[r] &    H_\m^{d+1}(Q)  \ar[r]  &  H_\m^{d+1}(B) \\
    H_\m^d(R) \ar[r] &    H_\m^d(S)=H_\n^d(S)  \ar[u] \ar[r]  & H_\m^{d+1}(Q) \ar[u]^\cong
} .\]
Chasing this diagram, it is easy to see that: $$\ker(H_\n^d(S)\to H_\n^d(T))\subseteq \im(H_\m^d(R)\to H_\n^d(S)).$$ This proves (\ref{equation--comparison of the sum of image with plus closure}). Finally, if $S$ satisfies the vanishing conditions for maps of local cohomology, then the left hand side of (\ref{equation--comparison of the sum of image with plus closure}) is $0$ by Definition \ref{definition--VCLC} (4), thus $0^+_{H_\n^d(S)}=0$.
\end{proof}

\begin{corollary}
\label{corollary--VCLC implies s.o.p plus closed}
If $(S,\n)$ satisfies the vanishing conditions for maps of local cohomology, then every ideal generated by a full system of parameters in $S$ is plus closed. In particular this implies $S$ is normal.
\end{corollary}
\begin{proof}
Let $I=(x_1,\dots,x_d)$ be any ideal generated by a full system of parameters of $S$. Consider the commutative diagram:
\[  \xymatrix{
    H_\n^d(S)\ar@{^{(}->}[r] & H_\n^d(S^+) \\
    S/I \ar[r] \ar@{^{(}->}[u] &    S^+/IS^+ \ar[u]
} .\]
The left vertical map is injective because $S$ is Cohen-Macaulay by Definition \ref{definition--VCLC}, and the map in the top row is injective by Lemma \ref{lemma--image contains the plus closure}. Chasing this diagram we know that $S/I\hookrightarrow  S^+/IS^+$ is injective. This proves that $I$ is plus closed.

Finally, every ideal generated by a system of parameters is plus closed implies that every ideal generated by part of a system of parameters is plus closed: suppose $(x_1,\dots,x_t)$ is part of a system of parameters, contained in $(x_1,\dots,x_t,x_{t+1},\dots,x_d)$. If $y\in (x_1,\dots,x_t)^+$, then $y\in (x_1,\dots,x_t,x_{t+1}^s,\dots,x_d^s)^+=(x_1,\dots,x_t,x_{t+1}^s,\dots,x_d^s)$ for every $s>0$. So \[y\in \bigcap_s(x_1,\dots,x_t,x_{t+1}^s,\dots,x_d^s)=(x_1,\dots,x_t).\]
In particular, we know that every principal ideal is plus closed. Let $y\in \overline{(x)}$, the integral closure of the ideal generated by $x$. Then $y\in \overline{(x)R^+}=(x)R^+$ because $R^+$ is integrally closed. So $y\in (x)^+=(x)$. This proves every principal ideal is integrally closed and hence $S$ is normal.
\end{proof}

\begin{lemma}
\label{lemma--VCLC implies pseudo-rational}
If $S$ satisfies the vanishing conditions for maps of local cohomology, then $S$ is pseudo-rational.
\end{lemma}
\begin{proof}
By our general assumption on commutative rings, $S$ is an excellent local domain hence is analytically unramified. By Definition \ref{definition--VCLC} and Corollary \ref{corollary--VCLC implies s.o.p plus closed}, $S$ is Cohen-Macaulay and normal. To check the last condition of pseudo-rationality, we let $W\to \Spec S$ be a proper birational map with $W$ normal, and we can assume this map is projective and birational by Chow's Lemma. Therefore $W\to \Spec S$ is just the blow up of some ideal $J$ in $S$, i.e.,
$$W=\Proj S\oplus Jt\oplus J^2t^2\oplus \cdots:=\Proj R.$$ Now we apply the Sancho de Salas exact sequence (see page 202 of \cite{SanchodeSalasBlowingupmorphismswithCohenMacaulayassociatedgradedrings}, or take cohomology of (\ref{equation--exact triangle}) in Section 5) to $W=\Proj R\to \Spec S$ to get ($d=\dim S$):
\[ \xymatrix{
H_E^{d-1}(W,O_W) \ar[r] & [H_{\n+R_{>0}}^d(R)]_0 \ar[r] \ar@{^{(}->}[d]  & H_\n^d(S) \ar[r]\ar[d]^\cong & H_E^d(W,O_W) \ar[r] & [H_{\n+R_{>0}}^{d+1}(R)]_0  \\
{}& H_{\n+R_{>0}}^d(R) \ar[r]^-0 & H_\n^d(S)  \\
}\]
Since $R$ has dimension $d+1$ and $S$ satisfies the vanishing conditions for maps of local cohomology, the bottom map is the zero map. By the commutativity of the above diagram, we see the map $[H_{\n+R_{>0}}^d(R)]_0 \to H_\n^d(S)$ vanishes. Therefore $H_\n^d(S) \to H_E^d(W,O_W)$ is injective. This finishes the proof.
\end{proof}

\begin{proposition}
\label{proposition--VCT implies VCLC and pseudo-rational}
When $S$ is a complete local domain, $S$ satisfies the vanishing conditions for maps of Tor implies $S$ satisfies the vanishing conditions for maps of local cohomology. Hence both implies $S$ has only pseudo-rational singularities.
\end{proposition}
\begin{proof}
Let $(R,\m)\twoheadrightarrow (S,\n)$ be a surjection with $R$ a domain. We may complete $R$ to get $\widehat{R}\twoheadrightarrow\widehat{S}=S$. Since $R$ is excellent by our general assumptions on commutative rings, $\widehat{R}$ is equidimensional. Since $S$ is a domain, the map $\widehat{R}\twoheadrightarrow S$ factors through $\widehat{R}\twoheadrightarrow R'\twoheadrightarrow S$ where $R'=\widehat{R}/P$ for $P$ a minimal prime of $\widehat{R}$. Thus in order to show $H_\m^j(R)\to H_\n^j(S)$ vanishes for $j<\dim R$, it suffices to show $H_\m^j(R')\to H_\n^j(S)$ vanishes for $j<\dim R'$. Hence without loss of generality, we may replace $R$ by $R'$ and assume $R$ is a complete local domain.

Now by Cohen's structure theorem, we have $(A,\m_0)\hookrightarrow (R,\m)$ a module-finite extension with $(A,\m_0)$ regular local. Let $E_A=E_A(A/\m_0)\cong H_{\m_0}^n(A)$ be the injective hull of the residue field of $A$. Since the \v{C}ech complex gives a flat resolution of $E_A$, we know that $\Tor_i^A(E_A, R)\cong H_\m^{n-i}(R)$ and $\Tor_i^A(E_A, S)\cong H_\m^{n-i}(S)$. Since $S$ satisfies the vanishing conditions for maps of Tor, by considering the map $A\to R\to S$, we have $\Tor_i^A(E_A, R)\to\Tor_i^A(E_A, S)$ vanishes for every $i\geq 1$. Hence $H_\m^j(R)\to H_\n^j(S)$ vanishes for $j<n=\dim R$. The last assertion then follows from Lemma \ref{lemma--VCLC implies pseudo-rational}.
\end{proof}

\begin{remark}
\label{remark--VCT implies VCLC for nice rings}
We assume $(S,\n)$ is complete in the proof of Proposition \ref{proposition--VCT implies VCLC and pseudo-rational} because we use Cohen's structure theorem to find $A\to R$ module-finite with $A$ regular. Hence the conclusion of Proposition \ref{proposition--VCT implies VCLC and pseudo-rational} still holds when we work with rings that are essentially of finite type over a field (we can use Noether normalization instead).
\end{remark}

\section{Vanishing conditions for maps of Tor and the splitting property}

The goal of this section is to prove $(1)\Leftrightarrow(3)$ in Theorem \ref{theorem--main theorem in equal characteristic}. As a corollary we will see that if $S$ satisfies the vanishing conditions for maps of Tor then $S$ is a splinter. We start with a lemma by restating (4.5) in \cite{HochsterHunekeApplicationsofbigCMalgebras}. This was stated only in the complete case in \cite{HochsterHunekeApplicationsofbigCMalgebras}, however the same argument works for rings essentially of finite type over a field (one needs to replace Cohen's structure theorem by Noether normalization in \cite{HochsterHunekeApplicationsofbigCMalgebras}).

\begin{lemma}[{\it cf.} (4.5) in \cite{HochsterHunekeApplicationsofbigCMalgebras}]
\label{lemma--reductions to check VCT}
Let $(S,\n)$ be either complete or essentially of finite type over a field. To show $(S,\n)$ satisfies the vanishing conditions for maps of Tor in a given characteristic, we may assume $A$ is local, $R$ is a domain, $A\to S$ is surjective and $M$ is finitely generated. Furthermore, it suffices to prove the vanishing of Tor for $i=1$.
\end{lemma}

\begin{remark}
\label{remark--shape of R when A to S is surjective}
Suppose $A\to R\to S$ satisfies that $A\to R$ is module-finite and torsion-free, and the composite map $A\to S$ is {\it surjective}. In this situation, if we set $S=A/P=R/\widetilde{P}$, then modulo $\widetilde{P}$, elements of $R$ come from elements of $A$. Thus in this case $R=A+\widetilde{P}$.
\end{remark}

\begin{theorem}
\label{theorem--(1)=(3) in main theorem}
Let $(S,\n)$ be either complete or essentially of finite type over a field.
\begin{enumerate}
\item $S$ satisfies the vanishing conditions for maps of Tor.
\item For every regular local ring $A$ with $S=A/P$, and every module-finite torsion-free extension $A\to B$ with $Q\in \Spec B$ lying over $P$, $P\to Q$ splits as $A$-modules.
\item For every regular local ring $A$ with $S=A/P$, every module-finite torsion-free extension $A\to B$ that splits as $A$-modules, and every $Q\in \Spec B$ lying over $P$, the map $P\to Q$ splits as $A$-modules.
\end{enumerate}
Then we have $(2)\Rightarrow(1)\Rightarrow(3)$. In particular, $(1)\Leftrightarrow(2)\Leftrightarrow(3)$ in equal characteristic.
\end{theorem}
\begin{proof}  Let $S=A/P$ with $A$ regular. Let $A\to B$ be a module-finite torsion-free extension with $Q\in \Spec B$ lying over $P$. We form the ring $R_0=A+Q\subseteq B$. Then $R_0$ is also a module-finite torsion-free extension of $A$ and we have $R_0/Q=A/P=S$. Now we look at the following commutative diagram:
\[  \xymatrix{
   0 \ar[r] & Q  \ar[r] &    R_0  \ar[r]  & S \ar[r] & 0\\
   0 \ar[r] & P  \ar[r]\ar[u]^\alpha &    A  \ar[u]^\beta \ar[r]  & S \ar[r]\ar[u]^\cong & 0
} .\]
Tensoring the above diagram with an arbitrary $A$-module $M$, we get:
\[  \xymatrix{
   \Tor_1^A(M, R_0) \ar[r]^{\varphi_M} & \Tor_1^A(M, S)\ar[r] & Q\otimes_AM  \ar[r] &    R_0\otimes_AM  \ar[r]  & S\otimes_AM \ar[r] & 0\\
   0 \ar[r] & \Tor_1^A(M, S) \ar[r]\ar[u]^\cong &  P\otimes_AM  \ar[r]\ar[u]^{\alpha\otimes \id_M} & A\otimes_AM  \ar[u]^{\beta\otimes \id_M}\ar[r] & S\otimes_AM \ar[r] \ar[u]^\cong & 0
} .\]
By a diagram chasing, one can see that
\begin{equation}
\label{equation--diagram chasing}
\alpha\otimes \id_M ~\mbox{is} ~\mbox{injective} \Longleftrightarrow \varphi_M=0 ~\mbox{and}~ \beta\otimes \id_M ~\mbox{is}~\mbox{injective}
\end{equation}

\vspace{0.5em}

$(1)\Rightarrow(3)$: Suppose $A\to B$ splits (i.e., we are in condition $(3)$), then $A\to R_0$ also splits, in particular $\beta\otimes \id_M$ is injective. Applying the vanishing conditions for maps of Tor to $A\to R_0\to S$, we know $\varphi_M=0$. Now (\ref{equation--diagram chasing}) implies $P\otimes_AM \xrightarrow{\alpha\otimes \id_M} Q\otimes_AM$ is injective for every $M$. But this implies $P\to Q$ splits by Corollary 5.2 in \cite{HochsterRobertsFrobeniusLocalCohomology} since $Q/P$ is a finitely generated $A$-module.

$(2)\Rightarrow(1)$: By Lemma \ref{lemma--reductions to check VCT}, we may assume $A\to R\to S$ satisfies $A\twoheadrightarrow S$ is surjective. Now we set $B=R$ and $Q=\ker(R\to S)$ in the above discussion. By Remark \ref{remark--shape of R when A to S is surjective}, we have $Q=\widetilde{P}$ and hence $R=A+\widetilde{P}=A+Q=R_0$ in this situation. Now (2) implies $P\to Q$ splits, in particular $\alpha\otimes \id_M$ is injective. Thus (\ref{equation--diagram chasing}) implies $$\varphi_M:  \Tor_1^A(M, R)=\Tor_1^A(M, R_0)\to \Tor_1^A(M, S)$$ vanishes for every $M$. This proves $S$ satisfies the vanishing conditions for maps of Tor, since it is enough to check the vanishing of Tor for $i=1$ by Lemma \ref{lemma--reductions to check VCT}.

Finally, in equal characteristic, every module-finite extension $A\to B$ splits when $A$ is regular. So $(2)\Leftrightarrow(3)$ and hence $(1)\Leftrightarrow(2)\Leftrightarrow(3)$.
\end{proof}

We next prove a lemma.

\begin{lemma}
\label{lemma--compatibly split}
Let $A\to B$ be a module-finite extension. Suppose $Q\in\Spec B$ lies over $P\in\Spec A$. If $P\to Q$ splits as $A$-modules and $\depth_PA\geq 2$, then $A\to B$ splits compatibly with $P\to Q$, i.e., there exists a splitting $\theta$: $B\to A$ such that $\theta(Q)=P$. In particular, $A/P\to B/Q$ splits as $A$-modules.
\end{lemma}
\begin{proof}
Let $\phi$: $Q\to P$ be a splitting. The exact sequences $0\to Q\to B\to B/Q\to 0$ and $0\to P\to A$ induce a commutative diagram:
\[  \xymatrix{
   \Hom_A(B, A)  \ar[r] &    \Hom_A(Q, A) \ar[r] & \Ext_A^1(B/Q, A) \\
   {}   &   \Hom_A(Q, P)  \ar@{^{(}->}[u] 
} .\]

Since $B/Q$ is a finitely generated $A$-module annihilated by $P$ and $\depth_PA\geq 2$, we know that $\Ext_A^1(B/Q, A)=0$ by Proposition 18.4 in \cite{Eisenbud95}. Hence $\Hom_A(B, A)$ maps onto $\Hom_A(Q, A)$, in particular it maps onto the image of $\Hom_A(Q, P)$. Thus there is a map $\theta$: $B\to A$ such that $\theta|_Q=\phi$. We show that $\theta$ has to be a splitting from $B$ to $A$. Suppose $\theta(1)=a\in A$, for every nonzero element $r\in P$, we have \[ra=r\theta(1)=\theta(r)=\phi(r)=r.\]
So $a=1$ and hence $\theta$ is a splitting from $B$ to $A$ such that $\theta(Q)=P$, i.e., $\theta$ compatibly splits $P\to Q$. Finally $\overline{\theta}$ gives a splitting from $B/Q\to A/P$. This finishes the proof.
\end{proof}

\begin{corollary}
\label{corollary--VCT implies splinter}
If $S$ satisfies the vanishing conditions for maps of Tor, then $S$ is a splinter.
\end{corollary}
\begin{proof}
We use a construction similar to the one used in the proof of Lemma \ref{lemma--image contains the plus closure}. We write $S=A/P$ such that $A$ is a regular local ring with $\depth_PA\geq 2$ (this can be achieved, for example, by adding indeterminants). Let $S\to T$ be a module-finite domain extension. Let $t_1,\dots,t_n$ be a set of generators of $T$ over $S=A/P$. Each $t_i$ satisfies a monic polynomial $f_i$ over $S$. We lift each $f_i$ to $A$ and form the ring $B=\frac{A[x_1,\dots,x_n]}{(f_1,\dots,f_n)}$. We have a natural surjection $B\twoheadrightarrow T$ with kernel $Q\in\Spec B$. It is straightforward to check that $Q$ lies over $P$.

Since $B$ is finite free over $A$, we know that $A\to B$ splits as a map of $A$-modules. Now applying $(1)\Rightarrow(3)$ of Theorem \ref{theorem--(1)=(3) in main theorem}, $P\to Q$ is split. Since $\depth_PA\geq 2$, by Lemma \ref{lemma--compatibly split}, $S=A/P \to B/Q=T$ splits as a map of $A$-modules (hence also as a map of $S$-modules). As this is true for any module-finite domain extension $T$ of $S$, $S$ is a splinter.
\end{proof}

\begin{remark}
Applying Corollary \ref{corollary--VCT implies splinter} with $S$ a regular local ring, we see that the vanishing conjecture for maps of Tor implies the direct summand conjecture in all characteristics. Although this is well known, we want to point out that the original proof given in \cite{HochsterHunekeApplicationsofbigCMalgebras} and \cite{Ranganathanthesis} depends on applying the vanishing conjecture to the map $A\to R\to S=R/\m$, i.e., studying the map from a mixed characteristic ring $(R,\m)$ to its residue field $S=R/\m$. Such a map, though being very natural, does {\it not} preserve the characteristic of the rings! Our above theorem gives a totally different proof, and it shows that the vanishing conjecture for maps of Tor, even if we restrict to $A\to R\to S$ all of the same characteristic, still implies the direct summand conjecture.
\end{remark}


\section{Main results}

In this section we prove our main theorem in equal characteristic. We begin by recalling some facts about dualizing complexes. For an integral scheme $X$, a {\it normalized dualizing complex} $\omega_X^\bullet$ is an object in $D^b_{\Coh}(X)$ which has finite injective dimension, such that the canonical map $O_X\to \mathbf{R}\sheafhom_{X}(\omega_X^\bullet, \omega_X^\bullet)$ is an isomorphism in $D^b_{\Coh}(X)$, and the first nonzero cohomology of $\omega_X^\bullet$ lies in degree $-\dim X$. Note that under this definition, if $\omega_X^\bullet$ is a normalized dualizing complex, then so is $\omega_X^\bullet\otimes\scr{L}$ for any line bundle $\scr{L}$ (in fact this is all the ambiguity for a connected scheme \cite{HartshorneResidues}).

To clear this ambiguity, notice that all our rings and schemes in this section are essentially of finite type over a field $k$ (or over a fixed scheme $\Spec S$ as in Theorem \ref{theorem--VCT implies derived splinters for Gorenstein rings} and Remark \ref{remark--new proof of splinter=derived splinter in char p}). Therefore we simply {\it define} $\omega_X^\bullet=\pi^{!}k$ (resp., $\pi^{!}\omega_S^\bullet$ for some chosen $\omega_S^\bullet$), where $\pi$: $X\to\Spec k$ (resp., $X\to \Spec S$) is the structural map and $\pi^{!}$ is the functor from Grothendieck duality theory \cite{HartshorneResidues}. By standard duality theory, $\omega_X^\bullet$ is a normalized dualizing complex discussed in the previous paragraph. Moreover, after this choice, we have $$\mathbf{R}\sheafhom _{X}(\mathbf{R}f_*O_Y, \omega_X^\bullet)\cong\mathbf{R}{f}_*\omega_Y^\bullet$$ for any proper and dominant morphism of integral schemes $f$: $Y\to X$, where $X$, $Y$ are both essentially of finite type over a field $k$ or over a fixed scheme $\Spec S$. We refer to \cite{HartshorneResidues} for details on Grothendieck duality theory and to Section 2.3 in \cite{BlickleSchwedeTuckerF-singularitiesvialterations} for a very nice summary.

Next we recall that for an excellent local domain $A$, a complex $$F_\bullet=0\to A^{b_n}\xrightarrow{\alpha_n} A^{b_{n-1}}\xrightarrow{\alpha_{n-1}}\cdots\xrightarrow{\alpha_2} A^{b_1}\xrightarrow{\alpha_1} A^{b_0}\to 0$$ of finitely generated free $A$-modules is said to {\it satisfy the standard conditions for rank and height (resp. rank and depth)} if, for every $1\leq i\leq n$, $\rank\alpha_i+\rank\alpha_{i+1}=b_i$ and the height (resp. the depth) of the ideal $I_{\rank{\alpha_i}}(\alpha_i)$ is at least $i$ where $I_t(\alpha_i)$ denotes the ideal generated by the size $t$ minors of a matrix for $\alpha_i$: it is independent of the choice of bases for $F_i$ and $F_{i-1}$ ($\rank\alpha_i$ is the largest integer $r$ such that $I_r(\alpha_i)\neq 0$). We use the convention that $I_0(\alpha_i)=R$ and the unit ideal has height infinity.

\begin{remark}
Let $F_\bullet$ be a complex of finite free $A$-modules. It is well known that $F_\bullet$ is acyclic (which means $F_\bullet$ is exact except possibly in degree $0$) if and only if $F_\bullet$ satisfies the standard conditions on rank and depth. Moreover, in characteristic $p>0$, $F_\bullet$ satisfies the standard conditions for rank and height if and only if $F_\bullet$ is {\it phantom acyclic}. We refer to \cite{HochsterHunekePhantomhomology} and \cite{Aberbachfinitephantomprojectivedimension} for details on phantom homology.
\end{remark}

Now we are ready to state and prove our key theorem that implies (and is in fact much stronger than) $(2)\Rightarrow(1)$ of Theorem \ref{theorem--main theorem in equal characteristic}.

\begin{theorem}[Key Theorem]\label{theorem--key theorem}
Let $(A,\m)$ be a local domain that is essentially of finite type over a field and $F_\bullet$: $0\to F_n\to \cdots\to F_1\to F_0\to 0$ be a complex of finitely generated free $A$-modules that satisfies the standard conditions on rank and height. Let $X\xrightarrow{f_0} Y\xrightarrow{f}\Spec A$ be maps of integral schemes such that $Y\to \Spec A$ is proper surjective and $X$ is a derived splinter. Then the natural map $$h^{-i}(F_\bullet\otimes\mathbf{R}f_*O_Y)\to h^{-i}(F_\bullet\otimes\mathbf{R}f_*\mathbf{R}{f_0}_*O_X)$$ induced by the pull back $f_0^*$ is the zero map for every $i>0$ (note that $F_i$ has cohomological degree $-i$).
\end{theorem}

\begin{proof}As the methods we use in characteristic $0$ and $p>0$ are very different, we separate the proof in two parts.

\vspace{0.5em}

\noindent{\it\textbf{Proof in characteristic $0$}}: We assume $(A,\m)$ is of equal characteristic $0$. Let $p$: $Z\to Y$ be a resolution of singularities and let $W=X\times_YZ$. We have the following diagram:
\[  \xymatrix{
   {} & Z \ar[d]^p\ar[ld]^g & W  \ar[l]^{f_1}\ar[d]^q\\
  \Spec A & Y \ar[l]^-f  & X  \ar[l]^{f_0}
} .\]

Since $p$ is a resolution of singularities, the map $q$ obtained by base change is proper and surjective. Because $X$ is a derived splinter, the natural map $q^*$: $O_X\to \mathbf{R}q_*O_W$ has a splitting $\eta$ in the derived category $D(\Coh(X))$, i.e., $\eta\circ q^*=\id$. Therefore we have the following commutative diagram in $D(\QCoh(\Spec A))$:
\[  \xymatrix{
    \mathbf{R}g_*O_Z \ar[r]^-{f_1^*} & \mathbf{R}g_*\mathbf{R}{f_1}_*O_W=\mathbf{R}f_*\mathbf{R}{f_0}_*\mathbf{R}q_*O_W \ar@/^1pc/@{.>}[d]^\eta\\
    \mathbf{R}f_*O_Y \ar[u]^{p^*} \ar[r]^-{f_0^*}& \mathbf{R}f_*\mathbf{R}{f_0}_*O_X  \ar[u]^{q^*}
} .\]

Now we tensor the above diagram with $F_\bullet$ in $D(\QCoh(\Spec A))$ and take cohomology in negative degree, we get a commutative diagram (since $F_\bullet$ is a complex of free $A$-modules, $\otimes^\mathbf{L}$ is the same as $\otimes$):

\[  \xymatrix{
    h^{-i}(F_\bullet\otimes \mathbf{R}g_*O_Z) \ar[r]^-{f_1^*} & h^{-i}(F_\bullet\otimes\mathbf{R}f_*\mathbf{R}{f_0}_*\mathbf{R}q_*O_W) \ar@/^1pc/@{.>}[d]^\eta\\
    h^{-i}(F_\bullet\otimes\mathbf{R}f_*O_Y) \ar[u]^{p^*} \ar[r]^-{f_0^*}& h^{-i}(F_\bullet\otimes\mathbf{R}f_*\mathbf{R}{f_0}_*O_X)  \ar[u]^{q^*}
} .\]
Since $\eta\circ q^*=\id$, in order to show $f_0^*$ induces the zero map, it is enough to show $\eta\circ f_1^*\circ p^*$ induces the zero map by the above commutative diagram. We will show this by proving that $h^{-i}(F_\bullet\otimes\mathbf{R}g_*O_Z)=0$ for every $i>0$, when $F_\bullet$ satisfies the standard conditions for rank and height and $Z\to \Spec A$ is proper surjective with $Z$ smooth.

We use induction on the dimension of $A$. When $\dim A=0$, $(A,\m)$ is Artinian and it is easy to see that every complex $F_\bullet$ that satisfies the standard conditions on rank and height is split exact except at the zeroth spot. Hence the complex $F_\bullet\otimes\mathbf{R}g_*O_Z\cong A^{n}\otimes\mathbf{R}g_*O_Z=\oplus^n \mathbf{R}g_*O_Z$ has no negative degree part, so $h^{-i}(F_\bullet\otimes\mathbf{R}g_*O_Z)=0$ for every $i>0$.

Now let $\dim A=d$. We set $G^\bullet=F_\bullet\otimes\mathbf{R}g_*O_Z$. Let $\underline{x}=x_1,\dots, x_d$ denote a full system of parameters of $A$ and let $C^\bullet(\underline{x}, A)$ be the \v{C}ech complex associated to $\underline{x}$. We also let $$\widetilde{G}^\bullet=G^\bullet\otimes^\mathbf{L}C^\bullet(\underline{x}, A)=G^\bullet\otimes C^\bullet(\underline{x}, A).$$
We compute $h^{-i}(\widetilde{G}^\bullet)$ using spectral sequences of the double complex:

\[  \xymatrix{
    {} & {} & {} & {}\\
    \cdots \ar[r] & G^{p+1}\otimes C^{q}(\underline{x}, A) \ar[r] \ar[u] &G^{p+1}\otimes C^{q+1}(\underline{x}, A)\ar[r]\ar[u]& \cdots\\
    \cdots \ar[r] & G^{p}\otimes C^{q}(\underline{x}, A) \ar[r]\ar[u] &G^{p}\otimes C^{q+1}(\underline{x}, A) \ar[r]\ar[u] & \cdots\\
    {} & {} \ar[u] &{} \ar[u] &
} \]
Each $C^{q}(\underline{x}, A)$ is a direct sum of localizations of $A$, in particular it is flat over $A$. So when we take the cohomology of the columns, we get $$E_1^{pq}=h^{p}(G^\bullet)\otimes C^{q}(\underline{x}, A).$$ Note that, when $q=0$, this is just $h^{p}(G^\bullet)$ while when $q>0$, this is a direct sum of proper localizations of $h^{p}(G^\bullet)$. However, when $p<0$, $h^{p}(G^\bullet)=h^{p}(F_\bullet\otimes\mathbf{R}g_*O_Z)$ is supported only at the maximal ideal $\m$ by the induction hypothesis. Because if it is supported at another prime ideal, say $P$, then $h^{p}((F_\bullet)_P\otimes\mathbf{R}g_*O_{Z\times_{\Spec R} \Spec R_P})\neq 0$. But $(F_\bullet)_P$ satisfies the standard conditions for rank and height as a free complex over $R_P$ (the ranks of all the $F_i$ are preserved and the height of an ideal does not decrease when we localize) and ${Z\times_{\Spec R}\Spec R_P}$ is smooth, we thus got a contradiction since $\dim R_P<d$. Hence we know that $E_1^{pq}=0$ when $p<0$ and $q>0$. In sum, the $E_1$-stage of the spectral sequence looks like:
\[  \xymatrix{
   \cdots\\
    E_1^{00}=h^0(G^\bullet)\ar[r] & E_1^{01} \ar[r] & E_1^{02} \ar[r]& \cdots \ar[r]& E_1^{0d} \\
    E_1^{-1, 0}=h^{-1}(G^\bullet) \ar[r]& E_1^{-1, 1}=0 \ar[r]  &E_1^{-1, 2}=0  \ar[r] & \cdots \ar[r] & E_1^{-1,d}=0\\
    E_1^{-2, 0}=h^{-2}(G^\bullet) \ar[r] & E_1^{-2, 1}=0 \ar[r]  &E_2^{-2, 2}=0  \ar[r] & \cdots \ar[r] &E_1^{-2, d}=0\\
    \cdots
} \]
From this we know that
\begin{equation}
\label{equation--degeneration of spectral sequence}
h^{-i}(\widetilde{G}^\bullet)=E_1^{-i, 0}=h^{-i}(G^\bullet)
\end{equation}
for every $i>0$. This is because $E_1^{-i, 0}=h^{-i}(G^\bullet)$ is the only nontrivial term that contributes to $h^{-i}(\widetilde{G}^\bullet)$ when $i>0$, and all the further differentials at this spot: $$E_r^{-i+r-1, -r}\to E_{r}^{-i, 0}\to E_r^{-i-r+1, r},$$ vanish since $E_r^{-i+r-1, -r}=E_r^{-i-r+1, r}=0$ when $i>0$ and $r\geq 1$.

Rewriting (\ref{equation--degeneration of spectral sequence}), we have:
\begin{equation}
h^{-i}(F_\bullet\otimes\mathbf{R}g_*O_Z)=h^{-i}(F_\bullet\otimes\mathbf{R}g_*O_Z\otimes C^\bullet(\underline{x}, A)).
\end{equation}
Since we have functorial isomorphism $\mathbf{R}\Gamma_\m(-)\xrightarrow{\cong}C^\bullet(\underline{x}, A)\otimes-$ in $D(\QCoh(\Spec A))$ by Proposition 3.1.2 in \cite{LipmanLecturesonlocalcohomologyandduality}. The above yields:
\begin{equation}
\label{equation--equation from the sqectral sequence argument}
h^{-i}(F_\bullet\otimes\mathbf{R}g_*O_Z)\cong h^{-i}(F_\bullet\otimes\mathbf{R}\Gamma_\m\mathbf{R}g_*O_Z).
\end{equation}
By local duality, we have that
\begin{equation}
\label{equation--local and Grothendieck duality}
h^j(\mathbf{R}\Gamma_\m\mathbf{R}g_*O_Z)=h^{-j}(\mathbf{R}\Hom(\mathbf{R}g_*O_Z, \omega_A^\bullet))^\vee=h^{-j}(\mathbf{R}g_*\omega_Z^\bullet)^\vee
\end{equation}
where $\omega_A^\bullet$, $\omega_Z^\bullet$ are the normalized dualizing complexes of $\Spec A$ and $Z$. Since $Z$ is smooth, $\omega_Z^\bullet\cong\omega_Z[n]$ where $n=\dim Z$. Hence $h^{-j}(\mathbf{R}g_*\omega_Z^\bullet)=h^{n-j}(\mathbf{R}g_*\omega_Z)=0$
when $n-j>n-d$ (equivalently, $j<d$) by Theorem 2.1 in \cite{KollarHigherdirectimagesofdualizingsheaves1}.\footnote{In \cite{KollarHigherdirectimagesofdualizingsheaves1}, the main theorem requires that the schemes are projective, but this is not essential. One can refer to Corollary 6.11 in \cite{EsnaultViehwegLecturesOnVanishing}.} Now from (\ref{equation--local and Grothendieck duality}), we know that
\begin{equation}
\label{equation--vanishing in j<d}
h^j(\mathbf{R}\Gamma_\m\mathbf{R}g_*O_Z)=0, \forall j<d.
\end{equation}

On the other hand, we know that $F_\bullet$ satisfies the standard conditions for rank and height. This implies $I_{\rank\alpha_n}(\alpha_n)$ must be the unit ideal when $n>d$ because there are no proper ideals in $R$ with height strictly bigger than the dimension of $R$. From this it follows that $F_\bullet$ is split exact at cohomological degree $-n$ when $n>d$ by Lemma 1 in \cite{BuchsbaumEisenbudWhatmakesacomplexexact}. Therefore, in $D(\Coh(\Spec A))$ or $D(\QCoh(\Spec A))$, $F_\bullet$ is quasi-isomorphic to a free complex $$H_\bullet: 0\to H_k\to H_{k-1}\to\cdots\to H_1\to H_0\to 0$$ with $k\leq d$ and $H_j$ has cohomological degree $-j$. Now from (\ref{equation--vanishing in j<d}), it is straightforward to check that $$h^{-i}(F_\bullet\otimes\mathbf{R}\Gamma_\m\mathbf{R}g_*O_Z)=h^{-i}(H_\bullet\otimes\mathbf{R}\Gamma_\m\mathbf{R}g_*O_Z)=0$$ for every $i>0$. Hence by (\ref{equation--equation from the sqectral sequence argument}) we know that $h^{-i}(F_\bullet\otimes\mathbf{R}g_*O_Z)=0$ for every $i>0$. This finishes our proof in equal characteristic $0$.

\vspace{1em}

\noindent{\it\textbf{Proof in characteristic $p>0$}}: Now we assume $(A,\m)$ is of equal characteristic $p>0$. By Theorem 1.5 in \cite{BhattDerivedsplintersinpositivecharacteristic}, there exists a finite surjective morphism $\pi$: $Z\to Y$ such that the pullback $\pi^*_{\geq 1}$: $\tau_{\geq 1}\mathbf{R}f_*O_Y\to\tau_{\geq 1}\mathbf{R}f_*\pi_*O_Z$ is the zero map. From this it follows (see the proof of Theorem 1.4 in \cite{BhattDerivedsplintersinpositivecharacteristic} or Lemma 5.1 in \cite{BlickleSchwedeTuckerF-singularitiesvialterations}) that the natural map $\mathbf{R}f_*O_Y\to \mathbf{R}f_*\pi_*O_Z$ factors through
\begin{equation}
\label{equation--factors through affine case}
\mathbf{R}f_*O_Y\xrightarrow{\theta} (f\circ\pi)_*O_Z\xrightarrow{\iota} \mathbf{R}f_*\pi_*O_Z.
\end{equation}

Let $g=f\circ\pi$. We know that $g_*O_Z$ is a module-finite extension of $A$, as $Z\to \Spec A$ is proper. Let $W=Z\times_YX$. We have the following commutative diagram:
\[  \xymatrix{
   \Spec(g_*O_Z)\ar[d] & Z \ar[l]\ar[d]^\pi \ar[ld]^g & W  \ar[l]^{f_1}\ar[d]^{}\\
  \Spec A & Y \ar[l]^-f  & X  \ar[l]^{f_0}
} .\]
This together with (\ref{equation--factors through affine case}) tell us that there is a commutative diagram in $D(\QCoh(\Spec A))$:
\begin{equation}
\label{equation--commutative diagram in derived category}
\xymatrix{
    g_*O_Z\ar[r]^-\iota &\mathbf{R}g_*O_Z \ar[r]^-{f_1^*} & \mathbf{R}f_*\mathbf{R}{f_0}_*\pi_*O_W \\ 
    {} & \mathbf{R}f_*O_Y \ar[lu]^\theta \ar[u]^{\pi^*} \ar[r]^-{f_0^*}& \mathbf{R}f_*\mathbf{R}{f_0}_*O_X  \ar[u]
} .
\end{equation}

Now we pick an arbitrary element $x\in h^{-i}(F_\bullet\otimes\mathbf{R}f_*O_Y)$ for an arbitrary $i>0$. We want to show that $x$ maps to $0$ in $h^{-i}(F_\bullet\otimes\mathbf{R}f_*\mathbf{R}{f_0}_*O_X)$. To prove this, we first look at the image of $x$ under the map $$h^{-i}(F_\bullet\otimes\mathbf{R}f_*O_Y)\xrightarrow{\theta} h^{-i}(F_\bullet\otimes g_*O_Z).$$ Let $y=\theta(x)$. Since we are in equal characteristic $p>0$ and $A\to g_*O_Z$ is a module-finite extension, $A^+=(g_*O_Z)^+$ is a balanced big Cohen-Macaulay algebra over $A$ \cite{HochsterHunekeInfiniteintegralextensionsandbigCMalgebras}. Since $F_\bullet$ satisfies the standard conditions for rank and height and $(g_*O_Z)^+$ is big Cohen-Macaulay, it follows from the generalized Buchsbaum-Eisenbud criterion (see Theorem 1.2.3 in \cite{Aberbachfinitephantomprojectivedimension}) that $h^{-i}(F_\bullet\otimes (g_*O_Z)^+)=0$ for $i>0$. In particular, we know that there exists a module-finite extension $B$ of $g_*O_Z$ such that the image of $y$ in $h^{-i}(F_\bullet\otimes B)$ is $0$. Let $W'=W\times_{\Spec (g_*O_Z)}\Spec B$. We know $\pi'$: $W'\to W\to X$ is a finite surjective map of schemes. Since $X$ is a derived splinter, in particular a splinter, we know that $O_X\to \pi'_*O_{W'}$ has a splitting $\eta$. In sum, after tensoring (\ref{equation--commutative diagram in derived category}) with $F_\bullet$ in $D(\QCoh(\Spec A))$ and taking cohomology, we get a commutative diagram:

\[  \xymatrix{
    h^{-i}(F_\bullet\otimes B) \ar[r] & h^{-i}(F_\bullet\otimes\mathbf{R}f_*\mathbf{R}{f_0}_*\pi'_*O_{W'}) \ar@/^7pc/@{.>}[dd]^\eta \\
    h^{-i}(F_\bullet\otimes g_*O_Z) \ar[u] \ar[r]^-{f_1^*\circ\iota} & h^{-i}(F_\bullet\otimes\mathbf{R}f_*\mathbf{R}{f_0}_*\pi_*O_W)\ar[u]\\
    h^{-i}(F_\bullet\otimes\mathbf{R}f_*O_Y) \ar[u]^{\theta} \ar[r]^-{f_0^*}& h^{-i}(F_\bullet\otimes\mathbf{R}f_*\mathbf{R}{f_0}_*O_X)  \ar[u]
} .\]

From this diagram, it is easy to see that the image of $x\in h^{-i}(F_\bullet\otimes\mathbf{R}f_*O_Y)$ maps to $0$ under $f_0^*$, because by our construction, the image of $x$ in $h^{-i}(F_\bullet\otimes B)$ is $0$. Since our choice of $x$ and $i>0$ are arbitrary, this proves that the map $f_0^*$: $h^{-i}(F_\bullet\otimes\mathbf{R}f_*O_Y)\to h^{-i}(F_\bullet\otimes\mathbf{R}f_*\mathbf{R}{f_0}_*O_X)$ is the zero map for every $i>0$. This finishes our proof in equal characteristic $p>0$.
\end{proof}

\begin{remark}
\label{remark--splinters and derived splinters are same in char p}
Note that in the above proof, in equal characteristic $p>0$, we only need to assume $X$ is a splinter. But splinters and derived splinters are the same in characteristic $p>0$ \cite{BhattDerivedsplintersinpositivecharacteristic}. In fact, in the course of our proof we use Theorem 1.5 of \cite{BhattDerivedsplintersinpositivecharacteristic}, which is a key ingredient in proving splinters and derived splinters are equivalent in characteristic $p>0$. We refer to \cite{BhattDerivedsplintersinpositivecharacteristic} for details.
\end{remark}

In the case of maps between rings instead of schemes, our Key Theorem \ref{theorem--key theorem} specializes to the following corollary:

\begin{corollary}
\label{corollary--ring version of key theorem}
Let $(A,\m)$ be a local domain that is essentially of finite type over a field and let $M$ be a finitely generated $A$-module of finite projective dimension. Let $A\to R\to S$ be ring homomorphisms such that $A\to R$ is a module-finite domain extension and $S$ is a derived splinter. Then the natural map:
$$\Tor_i^A(M, R)\to \Tor^A_i(M, S)$$
is the zero map for every $i>0$.
\end{corollary}
\begin{proof}
Since $M$ has a finite projective dimension, it has a finite free resolution $F_\bullet$. As $F_\bullet$ is acyclic, it satisfies the standard conditions for rank and depth and hence also the standard conditions for rank and height. Applying Theorem \ref{theorem--key theorem} to $F_\bullet$, $Y=\Spec R$, $X=\Spec S$ and noticing that there are no higher direct images because all the maps are affine, we find that $h^{-i}(F_\bullet\otimes R)\to h^{-i}(F_\bullet\otimes S)$ vanishes for every $i>0$. But $\Tor_i^A(M, R)=h^{-i}(F_\bullet\otimes R)$ and $\Tor_i^A(M, S)=h^{-i}(F_\bullet\otimes S)$, so the conclusion follows.
\end{proof}

Now we state and prove our main theorem.

\begin{theorem}
\label{theorem--main theorem}
Let $S$ be a local domain that is essentially of finite type over a field. The following are equivalent:
\begin{enumerate}
\item $S$ satisfies the vanishing conditions for maps of Tor.
\item $S$ is a derived splinter.
\item For every regular local ring $A$ with $S=A/P$ and every module-finite torsion-free extension $A\to B$ with $Q\in \Spec B$ lying over $P$, $P\to Q$ splits as $A$-modules.
\end{enumerate}
\end{theorem}
\begin{proof}
We already know $(1)\Leftrightarrow(3)$ by Theorem \ref{theorem--(1)=(3) in main theorem}. Moreover, recall that derived splinters are the same as rational singularities (equivalently, pseudo-rational singularities) in characteristic $0$, and are equivalent to splinters in characteristic $p>0$. Hence $(1)\Rightarrow(2)$ follows from Remark \ref{remark--VCT implies VCLC for nice rings} and Corollary \ref{corollary--VCT implies splinter} in characteristic $0$ and characteristic $p>0$ respectively. Finally we prove $(2)\Rightarrow (1)$. Suppose we have $A\to R\to S$ with $A$ regular and $R$ module-finite and torsion-free over $A$. To check $\Tor_i^A(M, R)\to \Tor_i^A(M, S)$ vanishes, we can assume $A$ is local, $R$ is a domain and $M$ is a finitely generated $A$-module by Lemma \ref{lemma--reductions to check VCT}. Since $A$ is regular, $M$ has finite projective dimension. Hence the vanishing of $\Tor_i^A(M, R)\to \Tor_i^A(M, S)$ follows immediately from Corollary \ref{corollary--ring version of key theorem}.
\end{proof}

\begin{remark}
\label{remark--key theorem extehds original vanishing theorems}
\begin{enumerate}
\item  We point out that in Theorem \ref{theorem--main theorem}, $(2)\Rightarrow (1)$ in characteristic $p>0$ also follows directly from the fact that $R^+$ is a balanced big Cohen-Macaulay algebra: one can use the same argument as in Theorem 4.1 in \cite{HochsterHunekeApplicationsofbigCMalgebras} and simply notice that the map $S\to S^+$ is always pure when $S$ is a splinter in characteristic $p>0$.
\item  However, our method in characteristic $0$ is of great interest: it gives the first proof of Theorem \ref{theorem--vanishing theorem for maps of Tor} (even in the regular case) in characteristic $0$ without using reduction to characteristic $p>0$. In fact, our result in characteristic $0$ doesn't even seem to follow from reduction to characteristic $p>0$. It is well known from \cite{SmithFRatImpliesRat} and \cite{HaraRatImpliesFRat} that a local ring essentially of finite type over a field of characteristic $0$ has rational singularities if and only if its mod $p$ reductions, for all sufficiently large $p$, are $F$-rational. But $F$-rationality is known to be weaker than being a derived splinter in characteristic $p>0$ and hence $F$-rational rings do not satisfy the vanishing conditions for maps of Tor in general by Theorem \ref{theorem--main theorem}.
\item  Moreover, equal characteristic regular local rings satisfying the vanishing conditions for maps of Tor is a very special case of our Key Theorem \ref{theorem--key theorem}: the case that $A$ is regular with $F_\bullet$ a free resolution of a finitely generated $A$-module $M$, $Y\to\Spec A$ is finite surjective and $X$ is regular affine. So our Theorem \ref{theorem--key theorem} greatly extends Hochster-Huneke's Theorem \ref{theorem--vanishing theorem for maps of Tor}, and actually it also generalizes the main theorems of \cite{HochsterHunekePhantomhomology}.
\end{enumerate}
\end{remark}

\begin{remark}
\label{remark--VCT generalizes Boutot's theorem}
We point out that Boutot's theorem that direct summands of rational singularities are rational singularities \cite{BoutotRationalsingularitiesandquotientsbyreductivegroups} follows from our vanishing conditions for maps of Tor applied to $M=E_A$, the injective hull of $A$. Let $(R,\m)\to (S,\n)$ be a split map of local rings essentially of finite type over a field of characteristic $0$ and let $S$ have rational singularities. For every surjection $(B,\m_1)\twoheadrightarrow (R,\m)$ with $B$ equidimensional, we can find $(A,\m_0)\to (B,\m_1)$ module-finite with $(A,\m_0)$ regular by Noether normalization. Now we consider the map $A\to B\to R\to S$. Since $S$ has rational singularities, it satisfies the vanishing conditions for maps of Tor by Theorem \ref{theorem--main theorem}. Hence $\Tor_i^A(E_A, B)\to\Tor_i^A(E_A, R)\to\Tor_i^A(E_A, S)$ vanishes for $i\geq 1$. This implies $\Tor_i^A(E_A, B)\to\Tor_i^A(E_A, R)$ vanishes for $i\geq 1$ because $R\to S$ splits. Hence $R$ satisfies the vanishing conditions for maps of local cohomology (recall that $\Tor_i^A(E_A, B)=H_\m^{d-i}(B)$). Therefore by Lemma \ref{lemma--VCLC implies pseudo-rational}, $R$ has rational singularities.
\end{remark}

As a consequence of Theorem \ref{theorem--main theorem}, we obtain a new characterization of rational singularities:

\begin{corollary}
\label{corollary--main corollary}
Let $(S,\n)$ be a local domain essentially of finite type over a field of characteristic $0$. Then $S$ has rational singularities if and only if for every regular local ring $A$ with $S=A/P$, every module-finite torsion-free extension $A\to T$, and every $Q\in\Spec T$ lying over $P$, the map $P\to Q$ splits as a map of $A$-modules.
\end{corollary}
\begin{proof}
This follows immediately from  $(2)\Leftrightarrow(3)$ in Theorem \ref{theorem--main theorem}, and the result that derived splinters are exactly rational singularities in equal characteristic $0$.
\end{proof}

We next want to prove a theorem that characterizes the vanishing conditions for maps of local cohomology in equal characteristic. We first prove a lemma that is of independent interest. Recall that in characteristic $p>0$, $0^*_{H_\n^d(S)}$ (the tight closure of $0$) is the largest proper submodule of ${H_\n^d(S)}$ that is stable under the natural Frobenius action \cite{SmithFRatImpliesRat}.

\begin{lemma}
\label{lemma--characterization of tight closure of 0 in top local cohomology}
Let $(S,\n)$ be a local domain of equal characteristic $p>0$. Then we have:
\begin{equation}
\label{equation--comparison of the sum of image with tight closure}
\sum_{R}\im(H_\m^d(R)\to H_\n^d(S))=0^*_{H_\n^d(S)}
\end{equation}
where the sum is taken over all $(R,\m) \twoheadrightarrow (S,\n)$ such that $\dim R/P>\dim S=d$ for every minimal prime $P$ of $R$.
\end{lemma}
\begin{proof}
Take a surjection $(R,\m)\to (S,\n)$, we first prove that the image of $H_\m^d(R)\to H_\n^d(S)$ is contained in $0^*_{H_\n^d(S)}$. Since $\dim R/P>d$ for every minimal prime $P$ of $R$, $R\to S$ obviously factors through $R\to R'\to S$ for some domain $R'$ with $\dim R'=d+1$. Hence the image of $H_\m^d(R)\to H_\n^d(S)$ is contained in the image of  $H_\m^d(R')\to H_\n^d(S)$, which is clearly a submodule of $H_\n^d(S)$ stable under the Frobenius action. Thus it is contained in $0^*_{H_\n^d(S)}$ as long as it is not equal to $H_\n^d(S)$. Therefore, it suffices to show that $H_\m^d(R')\to H_\n^d(S)$ is not surjective. Write $S=R'/Q$ for some height one prime ideal $Q$ of $R'$, we have the long exact sequence of local cohomology:
\begin{equation}
\label{equation--long exact sequence of non-injectivity of local cohomology}
\cdots \to H_\m^d(R')\to H_\n^d(S)\to H_\m^{d+1}(Q)\to H_\m^{d+1}(R').
\end{equation}
Since $R'$ has dimension $d+1$ and $Q$ is a height one prime,  $H_\m^{d+1}(Q)\to H_\m^{d+1}(R')$ is not injective by Lemma 3.3 in \cite{MaAsufficientconditionforFpurity}. Therefore $H_\m^d(R')\to H_\n^d(S)$ is not surjective by (\ref{equation--long exact sequence of non-injectivity of local cohomology}). Hence we have proved $$\sum_{R}\im(H_\m^d(R)\to H_\n^d(S))\subseteq 0^*_{H_\n^d(S)}.$$ On the other hand, Lemma \ref{lemma--image contains the plus closure} shows that that $$\sum_{R}\im(H_\m^d(R)\to H_\n^d(S))\supseteq 0^+_{H_\n^d(S)}=0^*_{H_\n^d(S)}$$ where the last equality follows from the main theorem of \cite{Smithtightclosureofparameterideals}. This finishes the proof.
\end{proof}

\begin{theorem}
\label{theorem--main theorem on VCLC}
Let $(S,\n)$ be a local domain that is essentially of finite type over a field. In characteristic $0$, $S$ satisfies the vanishing conditions for maps of local cohomology if and only if $S$ has rational singularities, while in characteristic $p>0$, $S$ satisfies the vanishing conditions for maps of local cohomology if and only if $S$ is $F$-rational.
\end{theorem}
\begin{proof}
The characteristic $0$ assertion follows (implicitly) from the proof of Theorem \ref{theorem--main theorem}, as $S$ satisfies the vanishing conditions for maps of local cohomology implies $S$ has rational singularities by Lemma \ref{lemma--VCLC implies pseudo-rational}. It remains to prove the characteristic $p>0$ statement. But this follows immediately from Lemma \ref{lemma--characterization of tight closure of 0 in top local cohomology} and Definition \ref{definition--VCLC}, since when $S$ is Cohen-Macaulay, $S$ is $F$-rational if and only if $0^*_{H_\n^d(S)}=0$ \cite{HochsterHunekeFRegularityTestElementsBaseChange}, \cite{SmithFRatImpliesRat}.
\end{proof}

Finally, it is quite natural to ask whether our main theorem, Theorem \ref{theorem--main theorem}, holds in mixed characteristic. By Theorem \ref{theorem--(1)=(3) in main theorem}, $(3)\Rightarrow(1)$ always holds and the main obstruction of $(1)\Rightarrow(3)$ is the direct summand conjecture in mixed characteristic. Below we give a partial answer of $(1)\Rightarrow(2)$. We believe this result and its proof is of independent interest (for example, see Remark \ref{remark--VCT implies DDSC} and \ref{remark--new proof of splinter=derived splinter in char p}).

\begin{theorem}
\label{theorem--VCT implies derived splinters for Gorenstein rings}
If $(S,\n)$ is a quasi-Gorenstein complete local domain (of mixed characteristic) that satisfies the vanishing condition for maps of Tor, then $S$ is a derived splinter.
\end{theorem}
\begin{proof}
We first note that the conditions imply $S$ is Cohen-Macaulay (and thus Gorenstein) by Proposition \ref{proposition--VCT implies VCLC and pseudo-rational} because $S$ is complete and satisfies the vanishing conditions for maps of Tor.

Let $\pi$: $X\to\Spec S$ be a proper surjective map, we want to show $S\to \mathbf{R}\pi_*O_X$ splits in the derived category of $S$-modules. By Chow's Lemma we may assume that $X$ is projective. I claim we may reduce to the case that $X\to \Spec S$ is a projective and generically finite map between integral schemes.\footnote{This should be well known. We provide the argument because we cannot find a good reference for this in mixed characteristic.} We first find an irreducible component $W$ of $X$ and an affine open $U=\Spec B$ of $W$ that dominates $\Spec S$. It follows that $B$ is a domain containing $S$, finitely generated as an $S$-algebra. Let $L$ be the fraction field of $S$,  we have $\dim (L\otimes B)=\dim B-\dim S$ by Theorem 13.8 in \cite{Eisenbud95}. Hence if $\dim B-\dim S\geq 1$, then $\dim (L\otimes B)\geq 1$. This means there exist nonzero primes in $B$ that contract to $0$ in $S$. Pick such a prime $Q$, we have $S\hookrightarrow B/Q$ is injective. Thus $V=\Spec{B/Q}$ still dominates $\Spec S$. Taking the closure of $V$ in $W$, call it $X'$, we have $X'\to \Spec S$ is projective and surjective with $\dim X'<\dim X$. We could repeat this process until we get $\dim X=\dim S$, i.e., $X\to \Spec S$ is projective and generically finite. Next we consider the Stein factorization:
\[X\to \Spec(\pi_*O_X)\to\Spec S.\]
Let $T=\pi_*O_X$. We know that $T$ is a module-finite domain extension of $S$. In particular, since $S$ is complete, this implies $T$ is a local ring and $\n T$ is primary to the maximal ideal of $T$. The map $X\to\Spec T$ is projective and birational, thus it is just the blow-up of some ideal $J\subseteq T$. Let $R=T[Jt]=T\oplus Jt\oplus J^2t^2\oplus\cdots$ and we have $X=\Proj R$.

Pick $f_1,\dots,f_n\in Jt=[R]_1$ such that $U=\{U_i=\Spec[R_{f_i}]_0\}$ is an affine open cover of $X$. We have an exact sequence of chain complexes (see page 150 of \cite{LipmanCMgradedalgebras}):
\[0\to \check{C}^\bullet(U, O_X)[-1]\to [C^\bullet(f_1,\dots,f_n, R)]_0\to T\to 0.\]
Since $\check{C}^\bullet(U, O_X)\cong\mathbf{R}\pi_*O_X$, the above sequence gives us (after rotating) an exact triangle:
\[[\mathbf{R}\Gamma_{R_{>0}}R]_0=[C^\bullet(f_1,\dots,f_n, R)]_0\to T\to \mathbf{R}\pi_*O_X\xrightarrow{+1}\]
Applying $\mathbf{R}\Gamma_{\n}$, we get:
\begin{equation}
\label{equation--exact triangle}
[\mathbf{R}\Gamma_{\n+R_{>0}}R]_0 \to  \mathbf{R}\Gamma_{\n}T  \to \mathbf{R}\Gamma_{\n}\mathbf{R}\pi_*O_X \xrightarrow{+1}.
\end{equation}

Let $d=\dim S=\dim T$ and $d+1=\dim R$. Next I prove two claims:
\begin{claim}
\label{first claim}
$[H_{\n+R_{>0}}^{d+1}(R)]_0=0$, thus $[\mathbf{R}\Gamma_{\n+R_{>0}}R]_0$ lives in degree $[0, 1, \dots, d]$.
\end{claim}
\begin{proof}
This is well known, because the $a$-invariant of the Rees ring is always $-1$ (for example, see 2.4.2 and 2.5.2 of \cite{HyrySmithOnnonvanishingconjectureKawamataandCore}). For the sake of completeness we point out that this also follows from (\ref{equation--exact triangle}). By local duality, the dual of $h^d(\mathbf{R}\Gamma_\n T)\to h^d(\mathbf{R}\Gamma_\n\mathbf{R}\pi_*O_X)$ is the natural inclusion $\pi_*\omega_X\hookrightarrow \omega_T$ (since $X\to \Spec T$ is birational). Hence $$[H_{\n+R_{>0}}^{d+1}(R)]_0=h^{d+1}([\mathbf{R}\Gamma_{\n+R_{>0}}R]_0)=0.$$
\end{proof}

\begin{claim}
\label{second claim}
There exists an $S$-linear surjection $\phi$: $T\twoheadrightarrow S$ such that the composite: $$[H_{\n+R_{>0}}^d(R)]_0\to H_\n^d(T)\xrightarrow{\phi} H_\n^d(S)$$ is the zero map (the first map is induced by the natural surjection $R\twoheadrightarrow T$).
\end{claim}
\begin{proof}
Let $R'=S\oplus Jt\oplus J^2t^2\oplus\cdots$ be the subring of $R$ (they only differ at the degree 0 spot). We note that since $J$ is a finitely generated $S$-module, $R'$ is a Noetherian graded domain over $S$. We have the following commutative diagram:
\[
 \xymatrix{
    0\ar[r] & R_{>0} \ar[r] & R \ar[r] & T \ar[r] & 0\\
    0\ar[r] & R'_{>0} \ar[u]^\cong \ar[r] & R'\ar[u] \ar[r] & S \ar[u]\ar[r] & 0
}
\]
Viewing everything as modules or algebras over $R'$, the above diagram induces a commutative diagram of local cohomology:
\[
 \xymatrix{
    [H_{\n+R_{>0}}^d(R)]_0 \ar[r]^-f & H_{\n}^d(T) \ar[r] & [H_{\n+R_{>0}}^{d+1}(R_{>0})]_0 \ar[r] & [H_{\n+R_{>0}}^{d+1}(R)]_0=0\\
    [H_{\n+R_{>0}}^d(R')]_0 \ar[u] \ar[r]^-0 & H_{\n}^d(S)\ar[u] \ar@{^{(}->}[r] & [H_{\n+R_{>0}}^{d+1}(R'_{>0})]_0 \ar[u]^\cong & {}
}
\]
The rightmost $0$ on the first line comes from Claim \ref{first claim} and the first map on the second line is $0$ because $S$ is complete and satisfies the vanishing conditions for maps of Tor, hence in particular it satisfies the vanishing conditions of local cohomology by Proposition \ref{proposition--VCT implies VCLC and pseudo-rational}. Chasing the diagram it follows immediately that $H_{\n}^d(S)\hookrightarrow H_{\n}^d(T)/\im(f)$. Since $S$ is quasi-Gorenstein, $H_\n^d(S)\cong E_S$ is an injective $S$-module. So there is a map $g$: $H_{\n}^d(T)/\im(f)\to H_{\n}^d(S)$ such that the composite: $$H_{\n}^d(S)\to H_\n^d(T)\to H_{\n}^d(T)/\im(f)\xrightarrow{g} H_\n^d(S)$$ is the identity. In particular, there is an splitting $H_{\n}^d(T)\xrightarrow{\phi} H_{\n}^d(S)$ of $H_{\n}^d(S)\hookrightarrow H_\n^d(T)$ such that the composite $[H_{\n+R_{>0}}^d(R)]_0\xrightarrow{f} H_\n^d(T)\xrightarrow{\phi} H_\n^d(S)$ is the zero map. However, it follows from the following commutative diagram:
\[
 \xymatrix{
    \Hom_S(H_\n^d(T), H_\n^d(S)) \ar[r]\ar[d]^\cong & \Hom_S(H_\n^d(S), H_\n^d(S))\ar[d]^\cong\\
    \Hom_S(T, S) \ar[r] & \Hom_S(S, S)
}
\]
that every splitting $H_{\n}^d(T)\xrightarrow{\phi} H_{\n}^d(S)$ comes from some surjection $T\xrightarrow{\phi} S$.
\end{proof}

Now we return to the proof of Theorem \ref{theorem--VCT implies derived splinters for Gorenstein rings}. I claim that the composite map: $$[\mathbf{R}\Gamma_{\n+R_{>0}}R]_0\to \mathbf{R}\Gamma_\n T\to H_\n^d(T)[-d]\xrightarrow{\phi}H_\n^d(S)[-d]\cong\mathbf{R}\Gamma_\n S$$ is the zero map \emph{in the derived category}: it induces zero on the $d$-th cohomology by Claim \ref{second claim}, but by Claim \ref{first claim}, $[\mathbf{R}\Gamma_{\n+R_{>0}}R]_0$ lives in degree $[0,1,...,d]$ while $H_\n^d(S)[-d]$ lives only in degree $d$, so the map is zero in the derived category. The last isomorphism follows because $S$ is Cohen-Macaulay.

Let $\phi$ be the surjection in Claim \ref{second claim}. There exits $t\in T$ such that $\phi(t)=1\in S$, in particular the composite $S\xrightarrow{\cdot t} T\xrightarrow{\phi} S$ is the identity. From the above discussion and (\ref{equation--exact triangle}), we have a natural diagram in the derived category of $S$-modules:

\[
 \xymatrix{
      [\mathbf{R}\Gamma_{\n+R_{>0}}R]_0  \ar@{.>}[rd]^0 \ar[r] &  \mathbf{R}\Gamma_{\n}T  \ar@/_/[d]_{\phi} \ar[r]  & \mathbf{R}\Gamma_{\n}\mathbf{R}\pi_*O_X \ar[r]^-{+1}& {}\\
      {}&  \mathbf{R}\Gamma_{\n}S \ar@/_/[u]_{\cdot t} & {}
}
\]
Taking Matlis dual and applying local duality, we get:
\begin{equation}
\label{equation--dualized exact triangle}
 \xymatrix{
      \mathbf{R}\pi_*\omega_X^\bullet  \ar[r] &  \omega_T^\bullet  \ar@/_/[d]_{\cdot t^\vee} \ar[r]  & [\mathbf{R}\Gamma_{\n+R_{>0}}R]_0^\vee \ar[r]^-{+1}& {}\\
      {}&  \omega_S^\bullet \ar@/_/[u]_{\phi^\vee}\ar@{.>}[ru]^0 & {}
}
\end{equation}
From (\ref{equation--dualized exact triangle}) it follows that $\phi^\vee$ factors through a map $\omega_S^\bullet\to \mathbf{R}\pi_*\omega_X^\bullet$ such that the composite: $$\omega_S^\bullet\to \mathbf{R}\pi_*\omega_X^\bullet\to \omega_T^\bullet\xrightarrow{\cdot t^\vee}\omega_S^\bullet$$ is the identity. Applying $\mathbf{R}\Hom_S(-, \omega_S^\bullet)$, we obtain: $$S\xrightarrow{\cdot t} T=\pi_*O_X\to \mathbf{R}\pi_*O_X\to S$$ such that the composite is the identity. But this implies $S\to \mathbf{R}\pi_*O_X\xrightarrow{\cdot t}\mathbf{R}\pi_*O_X\to S$ is the identity (the second map is induced by $O_X\xrightarrow{\cdot t} O_X$ viewing $t$ as a section of $X\to \Spec S$). Hence $S\to \mathbf{R}\pi_*O_X$ splits in the derived category of $S$-modules. Therefore $S$ is a derived splinter, as desired.
\end{proof}

At the moment we don't see how to drop the quasi-Gorenstein hypothesis on $S$ in Theorem \ref{theorem--VCT implies derived splinters for Gorenstein rings}, the subtle point seems to be Claim \ref{second claim}. However, the above result and its proof already have some interesting consequences.

\begin{remark}
\label{remark--VCT implies DDSC}
Since regular local rings are certainly quasi-Gorenstein, Theorem \ref{theorem--VCT implies derived splinters for Gorenstein rings} immediately implies that Hochster-Huneke's vanishing conjecture for maps of Tor (or equivalently, the strong direct summand conjecture \cite{Ranganathanthesis}) implies the derived direct summand conjecture of Bhatt \cite{BhattDerivedsplintersinpositivecharacteristic}.
\end{remark}

\begin{remark}
\label{remark--new proof of splinter=derived splinter in char p}
The argument used in Theorem \ref{theorem--VCT implies derived splinters for Gorenstein rings} gives a new proof that splinters and derived splinters are the same in characteristic $p>0$ for all local rings that are homomorphic image of Gorenstein local rings. First of all it is well known that splinters are Cohen-Macaulay in characteristic $p>0$ (we don't need completeness \cite{HunekeLyubeznikAbsoluteintegralclosureinpositivecharacteristic}, \cite{HochsterHunekeInfiniteintegralextensionsandbigCMalgebras}). Now the only place in the argument that we use the vanishing conditions for maps of Tor and the quasi-Gorenstein hypothesis seriously is in the proof of Claim \ref{second claim}. But this claim is clear in characteristic $p>0$ and we give a short argument as follows: by Theorem 2.1 of \cite{HunekeLyubeznikAbsoluteintegralclosureinpositivecharacteristic} we know that there exists a module-finite extension $B$ of $R$ such that the induced map $H_{\n+R_{>0}}^d(R)\to H_{\n+R_{>0}}^d(B)$ is zero. Since $B\otimes_RT$ is a module-finite extension of $T$ and hence a module-finite extension of $S$, the map $S\to B\otimes_RT$ splits as $S$-modules. Let $$\phi: T\to B\otimes_RT\xrightarrow{g} S$$ be the composite map for some splitting $g$: $B\otimes_RT\to S$. We have the following commutative diagram:
\[
 \xymatrix{
    H_{\n+R_{>0}}^d(B) \ar[r] & H_{\n}^d(B\otimes_RT) \ar[rd]^g & {}\\
    [H_{\n+R_{>0}}^d(R)]_0 \ar[u] \ar[r] & H_{\n}^d(T)\ar[u] \ar@{.>}[r]^\phi & H_{\n}^d(S)
}
\]
Since the left vertical map is the zero map by our choice of $B$, chasing through the diagram it is clear that the composite $[H_{\n+R_{>0}}^d(R)]_0\to H_\n^d(T)\xrightarrow{\phi} H_\n^d(S)$ is the zero map. Hence Claim \ref{second claim} holds in characteristic $p>0$ as long as $S$ is a splinter (without any quasi-Gorenstein or complete hypothesis).
\end{remark}

\section*{Acknowledgement}
It is a pleasure to thank Mel Hochster for several enjoyable discussions on the vanishing conjecture for maps of Tor and other homological conjectures. I want to thank Karl Schwede for reading a preliminary version of the paper and for his valuable comments. I would also like to thank Bhargav Bhatt and Anurag Singh for some helpful discussions. Finally I want to thank the referee, whose comments and suggestions lead to improvement of the paper.

\bibliographystyle{skalpha}
\bibliography{CommonBib}
\end{document}